\documentclass[11pt]{article}
\usepackage{latexsym,amsmath,amsthm,amssymb,graphicx,mathrsfs,amsfonts,yfonts,calrsfs,eufrak,bm,stackrel}
\input epsf
\usepackage{tikz}
\usepackage{relsize}
\usepackage{mathtools}
\DeclareFontEncoding{LS1}{}{}
\DeclareFontSubstitution{LS1}{stix}{m}{n}
\DeclareSymbolFont{symbols3}{LS1}{stixbb}{m}{n}
\DeclareMathSymbol{\bigslopedvee}{\mathbin}{symbols3}{"A7}

        {\qquad\hspace*{\fill}$\square$\end{trivlist}}

        {\qquad\hspace*{\fill}\end{trivlist}}

\newcounter{fig}

\newcounter{clai}
\newcounter{exa}
\newcounter{theorem}[section]
\renewcommand{\thetheorem}{\arabic{section}.\arabic{theorem}}
\newcounter{nona}[theorem]
\newcounter{nonanona}[theorem]

\renewcommand{\thenona}{\Alph{nona}}
\renewcommand{\thenonanona}{\alph{nonanona}}

\newenvironment{theorem}{\begin{trivlist}\item[]\refstepcounter{theorem}%
        {\bf\thetheorem\ Theorem}\par\nobreak\noindent\sl\ignorespaces}{%
        \ifvmode\smallskip\fi\end{trivlist}}

\newenvironment{noname}{\begin{trivlist}\item[]\refstepcounter{nona}%
        {\bf (\thenona)\ \ \ }\nobreak\noindent\sl\ignorespaces}{%
        \ifvmode\smallskip\fi\end{trivlist}}
\newenvironment{theoremplus}[1]{\begin{trivlist}\item[]%
        \refstepcounter{theorem}{\bf\thetheorem\ Theorem} {\rm (\,#1\,)}%
        \par\nobreak\noindent\sl\ignorespaces}{%
        \ifvmode\smallskip\fi\end{trivlist}}

\newenvironment{proposition}{\begin{trivlist}\item[]\refstepcounter{theorem}%
        {\bf\thetheorem\ Proposition}\par\nobreak\noindent\sl\ignorespaces}{%
        \ifvmode\smallskip\fi\end{trivlist}}
\newenvironment{conjecture}{\begin{trivlist}\item[]%
        \refstepcounter{theorem}{\bf\thetheorem\ Conjecture}\par%
        \nobreak\noindent\sl\ignorespaces}{%
        \ifvmode\smallskip\fi\end{trivlist}}
\newenvironment{conjectureplus}[1]{\begin{trivlist}\item[]%
        \refstepcounter{theorem}{\bf\thetheorem\ Conjecture} %
        {\rm(\,#1\,)}\par\nobreak\noindent\sl\ignorespaces}{%
        \ifvmode\smallskip\fi\end{trivlist}}

\newenvironment{observation}{\begin{trivlist}\item[]\refstepcounter{theorem}%
        {\bf\thetheorem\ Observation}\par\nobreak\noindent\ignorespaces}{%
        \ifvmode\smallskip\fi\end{trivlist}}

        {\endlist\unskip}

\newenvironment{claim}{\begin{trivlist}\item[]\refstepcounter{clai}%
        {\bf Claim \theclai}\mbox{ \ }\sl\ignorespaces}{%
        \ifvmode\smallskip\fi\end{trivlist}}

        {\endtrivlist}
\newcommand{\wideitem}[1]{\leavevmode\hangindent\mathindent\noindent%
        \hbox to \mathindent{#1\hfil}\ignorespaces}

\renewcommand{\emptyset}{\mathchar"001F}
\newcommand{\eop}{\rule{1.4ex}{1.4ex}}
\renewcommand{\H}{{\cal H}}
\newcommand{\A}{{\cal A}}
\newcommand{\B}{\mathcal{B}}
\newcommand{\C}{{\cal C}}

\renewcommand{\S}{{\cal S}}
\newcommand{\T}{{\cal T}}

\newcommand{\sms}{\vspace*{.2in}}

\title{\bf Cyclic Orderings of Paving Matroids}

\author{Sean McGuinness \\ Dept. of Mathematics\\ Thompson Rivers University\\
McGill Road, Kamloops BC\\ V2C5N3 Canada\\ email: smcguinness@tru.ca}

\date{}

\begin{document}

\maketitle

\begin{abstract}
A matroid $M$ of rank $r$ is {\it cyclically orderable} if there is a cyclic permutation of the elements of $M$ such that any $r$ consecutive elements form a basis in $M$.  An old conjecture of Kajitani,  Miyano, and Ueno states that a matroid $M$ is cyclically orderable if and only if for all $\emptyset \ne X \subseteq E(M),\ \frac {|X|}{r(X)} \le \frac {|E(M)|}{r(M)}.$   In this paper, we verify this conjecture for all paving matroids.

\end{abstract}

\section{Introduction}

A matroid $M$ of rank $r$ is {\bf cyclically orderable} if there is a a cyclic permutation of the elements of $M$ such that any $r$ consecutive elements is a base.

For a matroid $M$ and a subset $\emptyset \ne X\subseteq E(M),$ we define $\beta(X) := \frac {|X|}{r(X)},$ if $r(X) \ne 0$; otherwise, $\beta(X):= \infty.$  Let $\gamma(M) = \max_{\emptyset \ne X \subseteq E(M)}\beta(X).$

It turns out that the condition $\gamma(M) = \beta(E(M))$ is a necessary condition for a matroid $M$ to be cyclically orderable.  To see this, suppose $e_1e_2 \dots e_n$ is a cyclic ordering of a rank-$r$ matroid $M.$
Then for any nonempty subset $A \subseteq E(M),$ we have $r|A| = \sum_{i=1}^n |A \cap \{ e_i, e_{i+1}, \dots ,e_{i+r} \} | \le nr(A).$  The first equality follows from the fact that each element of $A$ appears in exactly $r$ sets $\{ e_i, e_{i+1}, \dots ,e_{i+r} \}$ and the second inequality follows from the fact that $|A \cap \{ e_i, e_{i+1}, \dots ,e_{i+r} \}| \le r(A).$
Consequently, $\beta(A) \le \beta (E(M))$ and hence $\gamma(M) = \beta(E(M)).$  
In light of this, the following conjecture of Kajitani, Miyano, and Ueno \cite{Kajetal} seems natural:

\begin{conjecture}
A matroid $M$ is cyclically orderable if and only if $\gamma(M) = \beta(E(M)).$ \label{con-Kaj}
\end{conjecture}

Despite having been around for decades, the above conjecture is only known to be true for a few special classes of matroids. In \cite{Bon}, the conjecture was shown to be true for sparse paving matroids.  Perhaps the strongest result thus far can be found in \cite{VanTho} where it was shown that Conjecture \ref{con-Kaj} is true when $r(M)$ and $|E(M)|$ are relatively prime. 

\begin{theoremplus}{Van Den Heuvel and Thomasse}
Let $M$ be a matroid for which $\gamma(M) = \beta(E(M)).$  If $|E(M)|$ and $r(M)$ are relatively prime, then $M$ has a cyclic ordering.\label{the-VanTho}
\end{theoremplus}

 It follows from recent results in \cite{BerSch} on {\it split matroids}, a class which includes paving matroids, that the conjecture is true for paving matroids $M$ where $|E(M)| \le 2 r(M).$ Coupled with Theorem \ref{the-VanTho}, we can replace $2r(M)$ by $2r(M)+1$ in this bound since $|E(M)|$ and $r(M)$ are relatively prime when $|E(M)| = 2r(M) +1.$ 
In this paper, we verify Conjecture \ref{con-Kaj} for all paving matroids.

\begin{theorem}
Let $M$ be a paving matroid where $\gamma(M) = \beta(E(M)).$  Then $M$ is cyclically orderable.\label{the-main}
\end{theorem}

For concepts, terminology, and notation pertaining to matroids, we shall follow Oxley \cite{Oxl} when possible.  For a matroid $M$, $\C(M)$ will denote the set of all circuits of $M.$

For a finite set $A$ and integer $k \le |A|$, we let ${\binom Ak}$ denote the set of all $k$-subsets of $A.$  For a collection of subsets $\A$ and integer $k$ we let ${\binom k{\A}}$ denote the set of all sets in $\A$ having cardinality $k.$

For a set $A$ and elements $x_1, \dots ,x_k$ we will often write, for convenience, $A + x_1 + x_2 + \cdots + x_k$ (resp. $A - x_1 - x_2 - \cdots -x_k$) in place of $A \cup \{ x_1, \dots ,x_k \}$ (resp, $A \backslash \{ x_1, \dots ,x_k \}$).

For a positive integer $n$, we let $[n]$ denote the set $\{ 1, \dots ,n \}.$

\subsection{Idea behind the proof}

To prove the main theorem, we shall use induction on $|E(M)|$.  To do this, we shall first remove a basis $S$ from $M$ so that the resulting matroid $M'$ satisfies $\gamma(M') = \beta(E(M) - S).$  While generally such a basis $S$ may not exist, we will show that such bases exist when $|E(M)|\ge 2r(M) +2.$  Applying the inductive assumption, $M'$ is cyclically orderable, with a cyclic ordering say $e_1e_2 \cdots e_m.$  We will show that for some $i\in [m]$ and some ordering of $S$, say $s_1s_2 \cdots s_r$ (where $r= r(M)$), the ordering $e_1 \cdots e_i s_1s_2 \cdots s_r e_{i+1} \cdots e_m$ is a cyclic ordering of $M$.   To give a rough idea of how to prove this, we will illustrate the proof in the case where $r(M) =3.$

Suppose $S = \{ s_1, s_2, s_3 \}$ is a basis of $M$ where $\gamma (M\backslash S) = \beta(E(M) -S)$ and $r(M\backslash S) = 3.$  Assume that $M' = M\backslash S$ has a cyclic ordering $e_1e_2 \cdots e_m$.  Suppose we try to insert the elements of $S$, in some order, between $e_m$ and $e_1$, so as to achieve a cyclic ordering for $M.$ Assume this is not possible. Then for every permutation $\pi$ of $\{ 1,2,3 \}$, $e_1e_2 \cdots e_m s_{\pi(1)}s_{\pi(2)}s_{\pi(3)}$ is not a cyclic ordering of $M.$  Thus for all permutations $\pi$ of $\{ 1,2,3 \}$, at least one of $\{ e_{m-1}, e_m, s_{\pi(1)} \}, \ \{ e_m, s_{\pi(1)}, s_{\pi(2)} \}, \{ s_{\pi(2)}, s_{\pi(3)}, e_1 \},$ or $\{ s_{\pi(3)}, e_1, e_2 \}$ is a circuit.  As an exercise for the reader, one can now show that there exist distinct $i,j \in \{ 1,2,3 \}$ such $\{ s_i, e_{m-1}, e_m \}, \{ s_j, e_1, e_2\}, S - s_i + e_m,$ and $S-s_j + e_1$ are circuits.  We may assume that $i=1$ and $j=2$. If instead, one were to assume that one could not insert the elements of $S$ in some order between $e_1$ and $e_2$ so as to achieve a cyclic ordering of $M$, then as above, there exist distinct $i',j'\in \{ 1,2,3 \},$ such that $\{ s_{i'}, e_{m}, e_1 \}, \{ s_{j'}, e_2, e_3\}, S - s_{i'} + e_1,$ and $S-s_{j'} + e_2$ are circuits. If $i' =1,$ then $\{ s_1, e_{m-1}, e_m \}$ and $\{ s_1, e_m, e_1 \}$ are circuits.  The circuit elimination axiom (together with the fact that $M$ is a paving matroid) would then imply that 
$( \{ s_1, e_{m-1}, e_m \} \cup \{ s_1, e_m, e_1 \}) - s_1 = \{ e_{m-1}, e_m, e_1 \}$ is a circuit, contradicting our assumption that $e_1e_2 \cdots e_m$ is a cyclic ordering of $M'.$  Also, if $i' = 2$, then $\{ s_2, e_m, e_1 \}$ and $\{ s_2, e_1, e_2 \}$ are circuits and hence by the circuit elimination axiom, $\{ e_m, e_1, e_2 \}$ is a circuit, a contradiction.  Thus $i' \not\in \{ 1, 2 \}$ and hence $i'=3$ and $\{ e_m, e_1, s_3 \}$ and $\{ s_1, s_2, e_2 \}$ are circuits.  Given that $\{ s_2, e_1, e_2 \}$ is also a circuit, it follows that 
$\{ e_1, e_2 \} \subset \mathrm{cl}(\{ s_1, s_2 \} ).$  Now $j'\in \{ 1,2 \},$ and $\{ s_{i'}, e_2,e_3 \}$ is a circuit, implying that $e_3 \in \mathrm{cl}(\{ s_1, s_2 \}).$  
However, this is impossible since (by assumption) $\{ s_1, s_2, s_3 \}$ is a basis. Thus there must be some ordering of $S$ so that when the elements of $S$ are inserted (in this order) between $e_m$ and $e_1$ or between $e_1$ and $e_2$, the resulting ordering is a cyclic ordering for $M.$ 

\section{Removing a basis from a matroid}
Let $M$ be a paving matroid where $\gamma(M) = \beta(E(M)).$
As a first step in the proof of Theorem \ref{the-main}, we wish to find a basis $B$ of $M$ where $\gamma(M\backslash B) = \beta(E(M)-B).$  Unfortunately, there are matroids where there is no such basis, as for example, the Fano plane.  In this section, we will show that, despite this, such bases exist when $|E(M)| \ge 2r(M) +2.$

The following is an elementary observation which we will refer to in a number of places.
\begin{observation}
For a basis $B$ in a matroid $M$ and an element $x\in E(M)-B$, the set $B+x$ has a unique circuit which contains $x$. \label{obs1}
\end{observation}

We will need the following strengthening of Edmonds' matroid partition theorem \cite{Edm} given in \cite{Fanetal}:

\begin{theorem}
Let $M$ be a matroid where $\gamma(M) = k + \varepsilon,$ where $k \in \mathbb{N}$ and $0 \le \varepsilon <1.$  Then $E(M)$ can be partitioned into $k+1$ independent sets with one set of size at most $\varepsilon r(M).$
\label{the0}
\end{theorem}

We are now in a position to prove the main result of this section.

\begin{proposition}
Let $M$ be a paving matroid where $\gamma(M) = \beta(E(M)),$ $|E(M)| \ge 2r(M) +2,$ and $r(M)\ge 3.$   Then there is a basis $B$ of $M$ where $\gamma(M\backslash B) = \beta(E(M)-B)$ and $r(M\backslash B) = r(M).$  \label{pro1}
\end{proposition}

\begin{proof}
Let $\gamma(M) = k + \frac {\ell}{r(M)}$ where $0 \le \ell < r(M)$ and $k\ge 2.$  Then $|E(M)| = kr(M) + \ell$ and it follows by Theorem \ref{the0} that one can partition $E(M)$ into $k$ independent sets $F_1, \dots ,F_k$ and one independent set $F_{k+1}$ having at most $\ell$ elements.  Since for all $i\in [k],\ |F_i| \le r(M)$ and $|F_{k+1}| \le \ell$ it follows that $kr(M) + \ell = |E(M)| = \sum_{i=1}^k |F_i| + |F_{k+1}| \le kr(M) + \ell.$  Thus equality must hold in the inequality and as such, for all $i\in [k],\ |F_i| = r(M)$ and $|F_{k+1}| = \ell.$  Thus $F_1, \dots ,F_k$ are bases in $M$.
Let $r = r(M).$ If $\ell = 0,$ then $|E(M)| = kr \ge 3r.$  In this case, we can take $B=F_k$ since for $M' = M\backslash F_k,$ it is seen that $\gamma(M') = k-1 = \beta(E(M')).$  Thus we may assume that $\ell >0.$ 

Let $F_k = \{ x_1, x_2, \dots ,x_r \}.$  Suppose there exist distinct $i,j \in [r]$ for which $r((F_k -x_i) \cup F_{k+1}) = r((F_k -x_j)\cup F_{k+1})=  r -1.$  Let $x\in F_{k+1}.$  Then $x+ (F_k -x_i)$ and $x + (F_k -x_j)$ are (distinct) circuits, contradicting Observation \ref{obs1}. Thus there is at most one $i\in [r]$ for which $r((F_k - x_i )\cup F_{k+1}) = r-1.$   As such, we may assume that for $i=1, \dots ,r-1$, $r((F_k - x_i) \cup F_{k+1}) = r.$  Thus for $i=1, \dots ,r-1$, there is a subset $A_i \subseteq F_k - x_i$ such that $B_i = A_i \cup F_{k+1}$ is a basis for $M.$

We shall show that the bases $B_i, \ i = 1, \dots ,r-1$ can be chosen so that for some $i\in[r-1]$, $B=B_i$ is a basis satisfying the proposition.  Suppose that none of the bases $B_i$ satisfy the proposition.  Then for all $i\in [r-1],$ there is a subset $X_i \subseteq E(M) - B_i$ for which $\beta(X_i) > \beta(E(M)-B_i).$  Since $k>1$, we have that $F_1 \subseteq E(M\backslash B_i)$ and hence $r(M\backslash B_i) = r.$  Thus we have
$\beta(E(M)-B_i) = k-1 + \frac {\ell}r.$  If $r(X_i) <r-1,$ then $X_i$ is independent and hence $\beta(X_i) =1 \le \beta(E(M)-B_i).$  Thus $r(X_i) \ge r-1$ and seeing as $\beta(X_i) > \beta(E(M) - B_i),$ we have
$r(X_i) \le r-1.$  Consequently, $r(X_i) = r-1$ and $\beta(X_i) = \frac {|X_i|}{r-1} > k-1 + \frac {\ell}r$.
Since $r(X_i) = r-1,$ it follows that for $j =1, \dots ,k-1$, $|X_i\cap F_j| \le r-1.$  Consequently, $|X_i| \le (k-1)(r-1) + \ell$.  If $|X_i| < (k-1)(r-1) + \ell,$ then $\beta(X_i) \le k-1 + \frac {\ell -1}{r-1}$, implying that $\beta(X_i) \le k-1 + \frac {\ell}r$, contradicting our assumptions. Thus it follows that $|X_i| = (k-1)(r-1) + \ell$ and for all $i\in [r-1]$ and for all $j\in [k-1],$ $|X_i\cap F_j| = r-1$, and $F_k - A_i \subset X_i$.
 Thus for all $i\in [r-1]$ and for all $j\in [k-1],$ $X_{ij} = X_i \cap F_j$ spans $X_i.$  Since all circuits in $M$ have size at least $r$, it follows that for all $j\in [k-1],$ and for all $x\in X_i - X_{ij},$  $X_{ij} +x$ is a circuit.

Suppose $k\ge 3.$ Let $i,j \in [r-1]$ where $i$ and $j$ are distinct (noting that such $i,j$ exists since $r\ge 3$).  Since $r\ge 3,$ there exists $x \in X_{i2} \cap X_{j2}.$ We have that $x + X_{i1}$ and $x + X_{j1}$ are circuits.  It follows by Observation \ref{obs1} that $X_{i1} = X_{j1}$ and thus $\mathrm{cl}(X_i) = \mathrm{cl}(X_j).$  Let $X = \mathrm{cl}(X_i).$ 
Since $F_k - A_i \subset X_i$, $F_k - A_j \subset X_j,$ $x_i \in F_k - A_i$ and $x_j \in F_k - A_j$, we have $\{ x_i, x_j \} \subset X$.   Since this applies to all $j\in[r-1]-i$, it follows that $F_k - x_r \subset X.$  If $r((F_k - x_r) \cup F_{k+1}) =r,$ then 
one could let $x_r$ play the role of $x_{r-1}$, and it would follow that $x_r \in X.$  This would imply that $F_k \subset X,$ an impossibility (since $r(X) = r-1$).
Thus $r((F_k -x_r) \cup F_{k+1}) = r-1.$  Given that $F_k -x_r \subset X$, we have $F_{k+1} \subseteq \mathrm{cl}(F_k - x_r) \subset X.$  Now it is seen that $\beta(X) = \frac {|X|}{r(X)} = \frac {k(r-1) + \ell}{r-1} = k + \frac {\ell}{r-1} > \gamma (M),$ a contradiction.

From the above, we have $k=2.$  Since $|E(M)| \ge 2r(M) +2,$ we have $\ell \ge 2.$ Let $i\in [r-1].$  

\begin{claim} For all $j\in [r-1] -i,$ one can choose $B_j$ so that $X_{j1} = X_{i1}.$ \label{claim1} \end{claim}
\begin{proof}
Let $j \in [r-1]-i$.  Suppose there exists $x \in (F_2 - A_i) \cap (F_2- A_j).$  Then $x \in X_i \cap X_j$ (since $F_2 - A_i \subset X_i$ and $F_2-A_j \subset X_j$) and, given that $r(X_i) = r(X_j) = r-1 = |X_{i1}| = |X_{j1}|$, it follows that $x + X_{i1}$ and $x+ X_{j1}$ are circuits.   It now follows by Observation \ref{obs1} that  $X_{i1}= X_{j1}.$   

Suppose instead that $(F_2 - A_i) \cap (F_2 - A_j) = \emptyset.$  That is, $F_2- A_i \subseteq A_j$ (and $F_2- A_j \subseteq A_i$).  Since $\ell \ge 2,$ there exists $x_s \in F_2 - A_j -x_j$. Now $x_s + B_j$ contains a (unique) circuit $C$ where $x_s \in C.$  We claim that $C\cap (F_2-A_i) \ne \emptyset.$  To see this, we observe that $|A_j - (F_2 - A_i)| = r - 2\ell.$  Thus 
\begin{align*}
|C \cap (F_2-A_i)| &= |C-x_s| - |C \cap ((A_j - (F_2 - A_i)) \cup F_3))| \\ 
&\ge |C| - 1 - ((r-2\ell)+\ell) = |C| -1 - r + \ell \ge \ell -1 \ge 1.\end{align*}   

Let $x_t \in C \cap (F_2 - A_i).$ 
Observing that $B_j - x_t + x_s$ is also a basis, let $A_j' = A_j - x_t+ x_s$
and $B_j' = B_j - x_t + x_s.$  Then $B_j' = A_j' + F_3$ and moreover, $x_t \in (F_2 - A_i) \cap (F_2 - A_j').$  Now defining $X_j$ as before, using $B_j'$ in place of $B_j$, one obtains that $X_{i1} = X_{j1},$ as in the previous case. \end{proof}

%
%
By the above claim, we may assume that for all $j\in [r-1] -i,$ the base $B_j$ can be chosen so that $X_{i1} = X_{j1}.$  Letting $X = \mathrm{cl}(X_{i})$ and following similar reasoning  as before,  we have that $(F_2 - x_r) \cup F_3 \subset X.$
Thus $\beta(X) = \frac {|X|}{r(X)} = \frac {2(r-1) + \ell}{r-1} = 2 + \frac {\ell}{r-1} > \gamma (M),$ a contradiction.  It follows that for some $i\in [r-1],$ the proposition holds for $B=B_i.$
\end{proof}


\section{$\mathbf{S}$-Pairs}
In the second part of the proof of Theorem \ref{the-main}, we will need to establish the existence of certain circuits.  More specifically, suppose $S$ is a basis as described in Proposition \ref{pro1} where we assume that $S= \{ s_1, \dots ,s_r \}.$
Suppose $e_1 e_2 \dots e_m$ is cyclic ordering for $M' = M \backslash S$ and our aim is to extend this ordering to a cyclic ordering for $M$ by inserting the elements of $S$, in some order, between $e_m$ and $e_1.$
Assuming this is not possible, it turns out (as in the case where $r(M) =3$) that there must be certain circuits.  For example, there are subsets $\{ B_1, B_2 \} \in {\binom S{r-2}}$ such that for all $s_i \in B_1, \ \{ s_i, e_{m-r+2}, \dots ,e_m \} \in \C(M)$ and
for all $s_i \in B_2, \{ s_i, e_1, \dots ,e_{r-1} \} \in \C(M).$  The results in this section and its successor,  lay the ground work to prove the existence of such circuits.
 
Let $S$ be a finite, nonempty set.  For $i=1,2,$ let $\S_i \subseteq  2^S.$  We call the pair $(\S_1, \S_2)$ an $\mathbf{S}$-{\bf pair}  if it has the following properties.

\begin{itemize}
\item[({\bf S1})] For $i=1,2,$  if $A,B \in \S_i$ where  $|A| = |B|+1$ and $B\subset A,$ then ${\binom A{|B|}} \subseteq \S_i.$   
\item[({\bf S2})] For $i=1,2$, if $A,B \in \S_i$ where $|A| = |B|$ and $|A\cap B| = |A|-1,$ then $A \cup B \in \S_i.$
\item[({\bf S3})] For $i=1,2,$ ${\binom S1} \not\subseteq \S_i$ and $S \not\in \S_i.$
\item[({\bf S4})] For $k=1, \dots ,|S|-1$, if ${\binom {S-x}k} \subseteq \S_1$ for some $x \in S,$ then ${\binom {S-x}{|S|-k}} \not\subseteq \S_2.$
\end{itemize}

In the next section, we shall need the following observations for an $S$-pair $(\S_1, \S_2)$ where $|S|=r.$

\begin{observation}
Let $A \subseteq S$ where $\alpha = |A|.$  Suppose that for some $i\in \{1,2 \}$ and some $j\in [\alpha],$ ${\binom Aj } \subseteq \S_i.$  Then for $k = j, \dots ,\alpha,$ ${\binom Ak} \subseteq \S_i.$\label{obs2}
\end{observation}

\begin{proof}
We may assume that $j < \alpha.$  Suppose that for some $k \in \{ j, \dots ,\alpha-1\},$ ${\binom Ak} \subseteq \S_i.$  Let $B \in {\binom A{k+1}}.$  Let $\{ b_1, b_2 \} \subseteq B$ and for $s=1,2,$  
let $B_s = B-b_s$.  By assumption, for $s =1,2$, $B_s \in \S_i.$ It now follows by ({\bf S2}) that $B = B_1 \cup B_2 \in \S_i.$  Consequently, we have that ${\binom A{k+1}} \subseteq \S_i.$
Arguing inductively, we see that for $k = j, \dots ,\alpha,$ ${\binom Ak} \subseteq \S_i.$  
\end{proof}

\begin{observation}
Let $A\in \S_i$ where $\alpha = |A|.$  Suppose that for some $j \in [\alpha -1]$ and $x\in A,$ we have ${\binom {A-x}j} \subseteq \S_i.$  Then ${\binom Aj} \subseteq \S_i.$\label{obs3}
\end{observation}

\begin{proof}
Suppose first that $j= \alpha-1.$ Then $A' =A -x \in \S_i.$ It follows by ({\bf S1}) that ${\binom A{\alpha -1}} \subseteq \S_i.$  Assume that $j < \alpha -1$ and the assertion holds for $ j +1;$ that is, if ${\binom {A-x}{j+1}} \subseteq \S_i$, then ${\binom A{j+1}} \subseteq \S_i.$
Suppose ${\binom {A-x}j} \subseteq \S_i.$  Then by Observation \ref{obs2}, ${\binom {A-x}{j+1}} \subseteq \S_i.$ Thus by assumption, ${\binom A{j+1}} \subseteq \S_i.$ Let $B \in {\binom Aj}$, where $x \in B.$  Let $y\in A-B$ and let $B' = B-x + y.$  Since $B' \in {\binom {A-x}j}$, it follows that $B' \in \S_1.$  However, we also have that $B+y \in \S_i.$  Thus it follows by ({\bf S1}) that $B \in \S_i.$   We now see that ${\binom Aj} \subseteq \S_i.$  The assertion now follows by induction.
\end{proof}

\begin{observation}
Let $A\subseteq S.$  Suppose for some $x \in A,$ $i\in \{ 1,2 \},$ and $j\ge 2,$ we have that $\{ B \in {\binom Aj} \ \big| \ x\in B \} \subseteq \S_i.$  Then ${\binom Aj} \subseteq \S_i$ and $A\in \S_i.$ \label{obs4}
\end{observation}

\begin{proof}
We may assume that $|A| \ge j+1.$  Let $B' \in {\binom {A-x}j}.$  Let $\{ y_1, y_2 \} \subseteq B'$ and for $s = 1,2$, let $B_s = B' - y_s + x.$  By assumption, $\{ B_1, B_2 \} \subset \S_i.$ It follows by ({\bf S2}) that $B =B' + x = B_1 \cup B_2 \in \S_i.$  Thus by ({\bf S1}) we have that ${\binom {B}j} \subseteq \S_i$ and hence $B'\in \S_i.$  It now follows that ${\binom Aj} \subseteq \S_i$, and moreover, $A\in \S_i$ (by Observation \ref{obs2}).
\end{proof}

\section{Order-consistent pairs}

Let $S = \{ s_1, s_2, \dots ,s_n \}$ be a set of $n$ elements and let $\S_1 \subseteq 2^S$ and $\S_2 \subseteq 2^S.$  We say that the pair $(\S_1, \S_2)$ is {\bf order-consistent} with respect to $S$ if for any permutation $\pi$ of $[n]$,  there exists $i \in [n]$ for which either
$\{ s_{\pi(1)} , \cdots ,s_{\pi(i)} \} \in \S_1$ or $\{ s_{\pi(i)}, \dots ,s_{\pi(n)} \} \in \S_2.$  Note that if $(\S_1, \S_2)$ is order-consistent, then $(\S_2, \S_1)$ is also order consistent.  To see this, let $\pi$ be a permutation of $[n]$
and let $\pi'$ be the permutation which is the reverse of $\pi$; that is, for all $i\in [n],$ $\pi'(i) = \pi(n-i+1).$  Since $(\S_1, \S_2)$ is order-consistent, there exists $i\in [n]$ such that either $\{ s_{\pi'(1)} , \cdots ,s_{\pi'(i)} \} \in \S_1$ or $\{ s_{\pi'(i)}, \dots ,s_{\pi'(n)} \} \in \S_2.$  Thus either $\{ s_{\pi(n-i+1)}, \dots ,s_{\pi(n)} \} \in \S_1$ or $\{ s_{\pi(1)}, \dots ,s_{\pi(n-i+1)} \} \in \S_2.$ Given that this holds for all permutations $\pi$, it follows that $(\S_2,\S_1)$ is an order-consistent pair.

Let $\Pi$ denote the set of all permutations of $[n]$ and let $\pi \in \Pi.$  
We say that a subset $A \in \S_1$ (resp. $B \in \S_2$) is $\mathbf{\pi}$-{\bf relevant} if there exists $i\in [n]$ such that $A = \{ s_{\pi(1)}, \dots ,s_{\pi(i)} \}$ (resp. $B = \{ s_{\pi(i)}, \dots ,s_{\pi(n)} \}$).  
Let $\Pi' \subseteq \Pi$ be a subset of permutations.  We say that a subset $\A \subseteq \S_1$ (resp. $\B \subseteq \S_2$) is $\mathbf{\Pi}'$-{\bf relevant} if for all $A \in \A$ (resp. $B \in \B$), there exists $\pi \in \Pi'$ such that $A$ (resp. $B$) is $\pi$-relevant.  We say that $(\A,\B)$ is order-consistent relative to $\Pi'$ if for all $\pi \in \Pi',$ either there exists $A\in \A$ for which $A$ is $\pi$-relevant, or there exists $B\in \B$ for which $B$ is $\pi$-relevant. 
%
For $i\in [n],$ we let $\Pi_i$ denote the set of permutations $\pi\in \Pi$ where $\pi(1) = i.$  
The following theorem will be instrumental in the proof of main theorem.

\begin{theorem}
Let $S= \{ s_1, \dots ,s_n \}$ be a set where $n\ge 3$ and let $(\S_1, \S_2)$ be an $S$-pair.  Then $(\S_1, \S_2)$ is order-consistent if and only if
there exists $(A_1, A_2) \in {\binom {n-1}{\S_1}}\times {\binom {n-1}{\S_2}}$, $A_1 \ne A_2$,
and $\{ B_1, B_2 \} \subset {\binom S{n-2}}$ where for $i=1,2$, $B_i \cap A_i = B_1 \cap B_2  \in  {\binom {A_1 \cap A_2}{n-3}}$ and
${\binom{B_i}1} \subset \S_i.$
\label{the1}
\end{theorem}

\begin{proof}
To prove sufficiency, suppose $A_i,B_i,\ i = 1,2$ are as described in the theorem.  Note that since $A_1 \ne A_2,$ we have $A_1 \cup A_2 = S.$  Also, since  $B_1 \cap B_2 \subseteq A_1 \cap A_2,$ we have $|B_1 \cap B_2| = n-3 = |A_1 \cap A_2| -1.$ 
Now $B_1 \not\subset A_1,$ for otherwise $|B_1 \cap B_2| = |A_1 \cap B_1| = |B_1| = n-2.$  Thus $B_1 \subseteq A_2,$ and likewise, $B_2 \subseteq A_1.$
  For $i=1,2,$ let $\T_i = \{ A_i \} \cup  {\binom {B_i}1}.$ We need only show that $(\T_1, \T_2)$ is order-consistent.  Suppose it is not.  Clearly it is order-consistent relative to the set of permutations $\pi$ for which $s_{\pi(1)} \in B_1$ or $ s_{\pi(n)} \in B_2.$   Let $\pi \in \Pi$ where $ s_{\pi(1)} \not\in B_1$ and $ s_{\pi(n)}  \not\in B_2.$  If $s_{\pi(1)} \not\in A_2,$ then $A_2 = \{ s_{\pi(2)}, \dots , s_{\pi(n)} \}$ and $A_2$ is $\pi$-relevant.  Thus  $A_2 - B_1 = \{ s_{\pi(1)} \} = (A_1 \cap A_2) - B_1.$  By similar reasoning, we also have $A_1-B_2 = \{ s_{\pi(n)} \} = (A_1 \cap A_2) - B_1.$  However, our assumptions imply that $(A_1 \cap A_2) - B_1 = (A_1 \cap A_2) - B_2$, and consequently, $s_{\pi(1)} = s_{\pi(r)}.$  This yields a contradiction.  It follows that $(\T_1, \T_2)$ is order-consistent.
%

To prove necessity, we shall  use induction on $n$. It is a straightforward exercise to verify the assertion for $n=3.$  We shall assume that $n \ge 4$ and the assertion is valid to all values less than $n$.  That is, if $|S| < n,$ and $(\S_1, \S_2)$ is an $S$-pair which is order-consistent, then there exist sets $A_i, B_i,\ i =1,2$ as described in the theorem.   Assume now that $S= \{ s_1, \dots ,s_n \}$ and $(\S_1, \S_2)$ is an $S$-pair which is order-consistent. 

 For all $k\in [n],$ let $S^k = S - s_k$ and let $\S_1^k = \{ A - s_k \ \big|\ A \in \S_1 \ \mathrm{and}\ s_k \in A \}$ and $\S_2^k = \{ A \in \S_2 \ \big| \ s_k \notin A \}.$
 We observe that properties ({\bf S1}) and ({\bf S2}) still hold for the pair $(\S_1^k, \S_2^k)$ whereas ({\bf S3}) and ({\bf S4}) may not. 

\begin{noname}
For all $k\in [n],$ one of the following holds:
\begin{itemize}
\item[({\bf a1})]  $\{ s_k \} \in \S_1.$
\item[({\bf a2})]  $S^k \in \S_2.$
\item[({\bf a3})]  ${\binom {S^k}1} \subseteq \S_2.$
\item[({\bf a4})]  For some $D \in {\binom{S^k}{n-2}},$ and positive integers $i,j$ where $i+j = n-1,$ ${\binom Di} \subseteq \S_1^k$ and ${\binom Dj} \subseteq \S_2^k .$
\item[({\bf a5})]  There exist 
$(A_1^k, A_2^k) \in {\binom {n-2}{\S_1^k}}\times {\binom {n-2}{\S_2^k}},$ $A_1^k \ne A_2^k$,
and $\{ B_1^k, B_2^k \} \subseteq {\binom {S^k}{n-3}}$ where for $i=1,2$, $B_i^k \cap A_i^k = B_1^k \cap B_2^k  \in  {\binom {A_1^k \cap A_2^k}{n-4}}$ and
${\binom{B_i^k}1} \subseteq \S_i^k.$

\end{itemize}\label{nona1}
\end{noname}

\begin{proof}
Let $k\in [n].$  Assume that none of ({\bf a1}) - ({\bf a4}) hold for $k.$  We will show that ({\bf a5}) must hold for $k.$ Clearly $S^k \not\in \S_1^k,$ for otherwise this would mean that $S \in \S_1$ which is not allowed by ({\bf S3}).  We also have that ${\binom {S^k}1} \not\subseteq \S_1^k$.  For if this was the case, then it would follow that for all $i \in [n] -k,$ 
$\{ s_i, s_k \} \in \S_1.$  It would then follow by Observation \ref{obs4} that $S\in \S_1$ violating ({\bf S3}).
Given that ({\bf a2}) -({\bf a4}) do not hold, $(\S_1^k, \S_2^k)$ is seen to be an $S^k$-pair.  Let $\pi \in \Pi_k$ and let $\pi' = \pi(2)\pi(3) \cdots \pi(n)$.  Since $(\S_1, \S_2)$ is order-consistent, there exists $A\in \S_1$ or $B \in \S_2$ and $i\in [n]$ such that either $A = \{ s_{\pi(1)}, \dots ,s_{\pi(i)} \}$ or
$B = \{ s_{\pi(i)}, \dots ,s_{\pi(n)} \}.$  Given that ({\bf a1}) and ({\bf a2}) do not hold, it follows that in the former case, $i\ge 2,$ $A' = \{ s_{\pi(2)},  \dots ,s_{\pi(i)} \} \in \S_1^k$ and hence $A'$ is $\pi'$-relevant. In the latter case, $i \ge 3$ and $B' = \{ s_{\pi(i)}, \dots ,s_{\pi(n)} \} \in \S_2^k$ and $B'$ is $\pi'$-relevant.  Given that $\pi$ was arbitrarily chosen from $\Pi_k$, we see that $(\S_1^k, \S_2^k)$ is order-consistent with respect to $S^k.$    By the inductive assumption,  there exist 
$(A_1^k, A_2^k) \in {\binom {n-2}{\S_1^k}}\times {\binom {n-2}{\S_2^k}},$ $A_1^k \ne A_2^k,$
and $\{ B_1^k, B_2^k \} \subset {\binom {S^k}{n-3}}$ where for $i=1,2$, $B_i^k \cap A_i^k = B_1^k \cap B_2^k  \in  {\binom {A_1^k \cap A_2^k}{n-4}}$ and
${\binom{B_i^k}1} \subset \S_i^k.$  Thus ({\bf a5}) holds for $k.$
\end{proof}

\begin{noname}
There is at most one integer $k$ for which ({\bf a2}) or ({\bf a3}) holds.  \label{nona1.25}
\end{noname}

\begin{proof}
It suffices to prove that ({\bf a2}) can hold for at most one integer $k$; if ({\bf a3}) holds for some integer $k$, then it follows by Observation \ref{obs2} that $S^k \in \S_2$, and hence ({\bf a2}) holds for $k.$
Suppose to the contrary that ({\bf a2}) holds for distinct integers $k$ and $\ell.$  Then $S^k \in \S_2$ and $S^\ell \in \S_2.$ It then follows by ({\bf S2}) that $S = S^k \cup S^\ell \in \S_2.$  However, this violates ({\bf S3}).  Thus no two such integers can exist.
\end{proof}

\begin{noname}
Property ({\bf a4}) holds for at most one integer $k.$ \label{nona1.5}
\end{noname}

\begin{proof}
Suppose ({\bf a4}) holds for distinct integers $k$ and $\ell.$
Then for some $i,j,i',j'$ where $i+j = n-1$, $i'+j' = n-1$, and subsets $D\in{\binom {S^k}{n-2}}$ and $D' \in {\binom {S^\ell}{n-2}},$ we have ${\binom Di} \subseteq \S_1^k$, ${\binom Dj} \subseteq \S_2^k,$ ${\binom {D'}{i'}} \subseteq \S_1^\ell,$ and ${\binom {D'}{j'}} \subseteq \S_2^\ell .$
By Observation \ref{obs4}, we have that $F_1 = D+ s_k \in \S_1$ and $F_2 = D' + s_\ell \in \S_1.$  If $F_1 \ne F_2,$ then by property ({\bf S2}), $F_1 \cup F_2 = S \in \S_1,$ violating ({\bf S3}).  Thus $F_1 = F_2 = S - s = S'$ for some $s \in S- s_k - s_\ell$ and $S' \in \S_1.$  

Let $i^* = \max \{ i, i' \}$ and $j^* = \min \{ j, j' \}.$  
 We claim that ${\binom {S'}{i^* +1}} \subseteq \S_1$ and ${\binom {S'}{j^*}} \subseteq \S_2.$  To prove the first assertion, we first note that it is true when $i^* = n-2$ since $S'  \in\S_1.$ 
 We may assume that $i^* < n-2.$  Then $i^* \le n-3 = |D\cap D'| = |S' -s_k -s_\ell |.$ Suppose first that $i^* = i.$
 Then by assumption, ${\binom {D\cap D'}{i^*}} \subset \S_1^k.$  Thus for all $X \in {\binom {S' - s_k-s_\ell}{i^*}}, \ X + s_k \in \S_1.$  It now follows by Observation \ref{obs4} that ${\binom {S'-s_\ell}{i^*+1}} \subseteq \S_1.$
 Now Observation \ref{obs3} implies that ${\binom {S'}{i^* +1}} \subseteq \S_1.$ Suppose now that $i^* > i.$  Then $i  < n-3$ and it follows by assumption that ${\binom {D\cap D'}{i}} \subseteq \S_1^k.$
 It now follows by Observation \ref{obs4} that ${\binom {D'}{i+1}} \subseteq \S_1.$  Also, since ${\binom {D}{i}} \subseteq \S_1^k,$ we have ${\binom {D}{i+1}} \subseteq \S_1.$
 Let $X \in {\binom {S'}{i+1}}.$  If $X \subseteq D$ or $X \subseteq D',$ then $X \in \S_1.$  Suppose neither occurs.  Then $\{ s_k, x_l \} \subseteq X$ and hence $X - s_k \in {\binom {D}{i}} \subseteq \S_1^k.$  It follows that $X \in \S_1.$  Consequently, ${\binom {S'}{i+1}} \subseteq \S_1.$  Since $i+1 \le i^*+1,$ it follows by Observation \ref{obs2} that ${\binom {S'}{i^*+1}} \subseteq \S_1.$

%
 
 To prove that ${\binom {S'}{j^*}} \subseteq \S_2,$ first suppose that $j^* = n-2$.  Then $j^* = j = j' = n-2.$   In this case, $D,D' \in \S_2$ and hence $S' = D \cup D' \in \S_2$ by ({\bf S2}). It would then follow by ({\bf S1}) that ${\binom {S'}{n-2}} \subseteq \S_2.$ Thus we may assume that $j^* < n-2.$ We have that ${\binom {D\cap D'}{j^*}} \subseteq \S_2.$  Given that $D\cap D' = S' - s_k - s_\ell,$ it follows by Observation \ref{obs3} that ${\binom {S'-s_\ell}{j^*}} \subseteq \S_2$ and this in turn implies that ${\binom{S'}{j^*}} \subseteq \S_2.$  
 
 Given that $i+j = i' + j' =n-1$, it follows that $i^* \le n-1 - j^*,$ and hence $i^* +1 + j^* \le n.$ By application of Observation \ref{obs2}, we have that ${\binom {S'}{n -i^* -1}} \subseteq \S_2$.
However, we now have both ${\binom {S'}{i^*+1}} \subseteq \S_1$ and ${\binom {S'}{n -i^* -1}} \subseteq \S_2$, violating ({\bf S4}).  We conclude that ({\bf a4}) can hold for at most one integer $k$. 
\end{proof}

\begin{noname}
There exists $T \in {\binom S{n-3}}$ such that either ${\binom T1} \subseteq \S_1$ or ${\binom T1} \subseteq \S_2.$ \label{nona2}
\end{noname}

\begin{proof}
Assume that there is no subset $T\in  {\binom S{n-3}}$ such that ${\binom T1} \subseteq \S_1$.  Then there are at least three integers $k$ for which ({\bf a1}) does not hold.   
By ({\bf\ref{nona1.25}}) and ({\bf \ref{nona1.5}}), ({\bf a2}) or ({\bf a3}) holds for at most one integer $k$ and ({\bf a4}) holds for at most one integer $k.$
Thus there exists $k\in[n]$ such that none of ({\bf a1}) - ({\bf a4}) hold.  By ({\bf \ref{nona1}}), ({\bf a5}) holds for $k.$
Thus there exists $(A_1^k, A_2^k) \in {\binom {n-2}{\S_1^k}}\times {\binom {n-2}{\S_2^k}}$, $A_1^k \ne A_2^k,$
and $\{ B_1^k, B_2^k \} \subset {\binom {S^k}{n-3}}$ where for $i=1,2$, $B_i^k \cap A_i^k = B_1^k \cap B_2^k  \in  {\binom {A_1^k \cap A_2^k}{n-4}}$ and
${\binom{B_i^k}1} \subset \S_i^k.$ 
Thus we see that ${\binom {B_2^k}1} \subseteq \S_2^k \subseteq \S_2.$  This completes the proof.
\end{proof}

\begin{noname}
There exists $T \in {\binom S{n-2}}$ such that either ${\binom T1} \subseteq \S_1$ or ${\binom T1} \subseteq \S_2.$ \label{nona3}
\end{noname}

\begin{proof}
By ({\bf \ref{nona2}}), there exists $T \in {\binom S{n-3}}$ such that either ${\binom T1} \subseteq \S_1$ or ${\binom T1} \subseteq \S_2.$  We claim that it suffices to prove the assertion when ${\binom T1} \subseteq \S_1.$
For if instead ${\binom T1} \subseteq \S_2$, then redefine $\S_i^k$ so that for all $k\in [n]$,  $\S_1^k = \{ A \in \S_1 \ \big| \ s_k \not\in A \}$ and $\S_2^k = \{ A - s_k \ \big| \ A \in \S_2 \ \mathrm{and}\ s_k \in A \}.$  Now it is seen that  ({\bf \ref{nona1}}) - ({\bf \ref{nona1.5}}) still hold when in ({\bf a1}) - ({\bf a5}), we switch $\S_1$ with $\S_2$ and switch $\S_1^k$ with $\S_2^k$.  Now one can use the same proof as in the case when ${\binom T1} \subseteq \S_1$.

By the above, we may assume that ${\binom T1}\subseteq \S_1.$  Furthermore, we may assume that $T = \{ s_1, \dots ,s_{n-3} \}.$   Next, we will show that either $\{ s_i \} \in \S_1$ for some $i\in \{ n-2, n-1, n \},$ or ${\binom {S'}1} \subseteq \S_2$ for some $S' \in {\binom {S}{n-2}}.$
We may assume that ({\bf a1}) and ({\bf a3}) do not hold for all $k \in \{ n-2, n-1,n \}.$  Furthermore, by ({\bf \ref{nona1.25}}) and ({\bf \ref{nona1.5}}), ({\bf a2}) holds for at most one integer $k\in \{ n-2,n-1,n \}$ as does ({\bf a4}).   As such, we may assume that ({\bf a2}) and ({\bf a4}) do not hold for $k=n-2.$  Thus by By ({\bf \ref{nona1}}), ({\bf a5}) holds for $k=n-2$.
Thus there exist 
$(A_1^{n-2}, A_2^{n-2}) \in {\binom {n-2}{\S_1^{n-2}}}\times {\binom {n-2}{\S_2^{n-2}}}$, $A_1^{n-2} \ne A_2^{n-2},$
and $\{ B_1^{n-2}, B_2^{n-2} \} \subset {\binom {S^{n-2}}{n-3}}$ where for $i=1,2$, $B_i^{n-2} \cap A_i^{n-2} = B_1^{n-2} \cap B_2^{n-2}  \in  {\binom {A_1^{n-2} \cap A_2^{n-2}}{n-4}}$ and
${\binom{B_i^{n-2}}1} \subset \S_i^{n-2}.$ 

Suppose $s_i \in B_1^{n-2} \cap \{ s_1, \dots ,s_{n-3} \}.$  By assumption, $\{ s_i \} \in \S_1.$  However, given that $s_i \in B_1^{n-2}$, we also have that $\{ s_i \} \in \S_1^{n-2}$ and hence $\{ s_i, s_{n-2} \} \in \S_1.$  By ({\bf S1}), $\{ s_{n-2} \} \in \S_1,$ a contradiction.
Thus  $B_1^{n-2} \cap \{ s_1, \dots ,s_{n-3} \} = \emptyset$ and hence $B_1^{n-2} \subseteq \{ s_{n-1}, s_n \}.$ Consequently, 
$n-3 \le 2$ and hence $n \le 5.$  To complete the proof, we need only consider two cases:

\sms

\noindent {\bf Case 1}:  $n=5$.

\sms

We have $B_1^{n-2} = B_1^3 = \{ s_4, s_5 \}.$  We may assume that $A_2^3 = \{ s_1, s_4, s_5 \},$ where $B_1^3 \cap B_2^3 = \{ s_4 \}.$ Thus $A_1^3 = \{ s_1, s_2, s_4 \}$ and $B_2^3 = \{ s_2, s_4 \}.$  Then $A_1^3 + s_3  = \{ s_1, s_2, s_3, s_4 \} \in \S_1$ and  
$A_2^3 \in \S_2.$  Given that ${\binom {B_2^3}1 } \subseteq \S_2,$ we may assume that for all $i\in \{ 1, 3, 4 \},$ $\{ s_i \} \not\in \S_2.$
Since by assumption ({\bf a1}) and ({\bf a3}) do not hold for $k\in \{ 3,4,5 \},$ it follows by ({\bf \ref{nona1}}) that for all $k \in \{ 4,5 \},$ one of ({\bf a2}), ({\bf a4}), or ({\bf a5}) must hold.  

Suppose ({\bf a5}) holds for $k=5$.  Then arguing as above, we have that $B_1^5 = \{ s_3, s_4 \}$ and hence $A_1^5 = \{ s_1, s_2, s_3 \}$ or $A_2^5 = \{ s_1, s_2, s_4 \}.$  Thus
either $\{ s_1, s_2, s_3, s_5 \} \in \S_1$ or $\{ s_1, s_2, s_4, s_5 \} \in \S_1$. Given that $\{ s_1, s_2, s_3, s_4 \} \in \S_1,$ it would follow by {\bf (S2)} that $S\in \S_1$, contradicting {\bf (S3)}.  Thus ({\bf a5}) does not hold for $k=5$.
%
  
Suppose  ({\bf a4}) holds for $k=5.$ 
%
Then there exists a subset $D' \in {\binom {S^5}3}$ and integers $i,j$ where $i+j = 4$ such that ${\binom {D'}i} \subseteq \S_1^5$ and ${\binom {D'}j} \subseteq \S_2^5.$  Let $D = D' + s_5.$  By Observation \ref{obs4}, it follows that ${\binom {D}{i+1}} \subseteq \S_1$ and $D \in \S_1.$
Clearly $D \ne \{ s_1, s_2, s_3, s_4 \}$ and hence it follows by property ({\bf S2}) that $D \cup \{ s_1, s_2, s_3, s_4 \} = S \in \S_1,$ yielding a contradiction.  Thus ({\bf a2}) holds for $k=5$ and hence $\{ s_1, s_2, s_3, s_4 \} \in \S_2$.  By ({\bf \ref{nona1}}) and ({\bf \ref{nona1.25}}) it follows that either ({\bf a4}) or ({\bf a5}) holds for $k=4.$  

Suppose ({\bf a5}) holds for $k=4.$  Arguing as before, we see that $B_1^4 = \{ s_3, s_5 \}$ and either $A_1^4 = \{ s_1, s_2, s_3 \}$ or $A_1^4 = \{ s_1, s_2, s_5 \}.$  In the latter case, we have that $\{ s_1, s_2, s_4, s_5 \} \in \S_1$.  It would then follow by ({\bf S2}) that $\{ s_1, s_2, s_3, s_4 \} \cup \{ s_1, s_2, s_4, s_5 \} = S \in \S_1$, contradicting ({\bf S3}).  Thus we have that $A_1^4 = \{ s_1, s_2, s_3 \}$.  It now follows that $\{ s_3 \} = A_1^4 \cap B_1^4 = B_1^4 \cap B_2^4.$  Thus $s_3 \in B_2^4,$ implying that $\{ s_3 \} \in \S_2$, contradicting our assumptions.

By the above, ({\bf a4}) must hold for $k=4.$ Thus there exists a subset $D' \in {\binom {S^4}3}$ and integers $i,j$ where $i+j = 4$ such that ${\binom {D'}i} \subseteq \S_1^4$ and ${\binom {D'}j} \subseteq \S_2^4.$ 
It follows by Observation \ref{obs4} that for $D = D' +s_4,$ ${\binom D{i+1}} \subseteq \S_1$ and $D\in \S_1.$  If $D \ne \{ s_1, s_2, s_3, s_4 \},$ then we would have $D \cup \{ s_1, s_2, s_3, s_4 \} = S \in \S_1,$ contradicting ({\bf S3}). Thus $D = \{ s_1, s_2, s_3, s_4 \}$ and consequently, $D' = \{ s_1, s_2, s_3 \}.$ Given that ${\binom {D'}j} \subseteq \S_2$ and $\{ s_3 \} \not\in \S_2,$ it follows that $j\ge 2.$ 

Suppose $i=1.$  Then ${\binom D2} \subseteq \S_1.$  Given that ${\binom {B_1^3}1} \subseteq \S_1^3,$ it follows that $\{ s_5 \} \in \S_1^3$ and hence $\{ s_3, s_5 \} \in \S_1.$  Thus we have $\{ B \in {\binom S2} \ \big| \ s_3 \in B \} \subseteq \S_1.$  It now follows by Observation \ref{obs4} that $S\in \S_1,$ contradicting ({\bf S3}). Thus $i\ge 2$ and $i=j=2.$ We now have that ${\binom D3} \subseteq \S_1.$ Given that $B_2^3 = \{ s_2, s_4 \} \in \S_2$ and ${\binom {D'}2} \subseteq \S_2,$ it follows $\{ B \in {\binom {D}2} \ \big| \ s_2 \in B \} \subseteq \S_2.$ Thus by Observation \ref{obs4}, we have ${\binom D2} \subseteq \S_2.$ However, we now have both ${\binom D3} \subseteq \S_1$ and ${\binom D2} \subseteq \S_2$, contradicting ({\bf S4}).
This completes the case $n=5.$

\sms

\noindent {\bf Case 2}: $n=4$.

\sms

We may assume that  $B_1^{n-2} = B_1^2 = \{ s_4 \}, \ A_1^2 = \{ s_1, s_3 \}.$ 
There are two possible cases to consider for $A_2^2$ and $B_2^2$:  either $A_2^2 = \{ s_1, s_4 \}$ and $B_2^2 = \{ s_3 \}$ or $A_2^2 = \{ s_3,  s_4 \}$ and $B_2^2 = \{ s_1 \}.$ We shall assume the former -- the latter case can be handled similarly.  
We have that $A_1^2 = \{ s_1, s_3 \}$ and hence $A_1^2 + s_2 = \{ s_1, s_2, s_3 \} \in \S_1$ and $B_1^2 + s_2 = \{ s_2, s_4 \} \in \S_1.$  We also have that $A_2^2 = \{ s_1, s_4 \} \in \S_2$ and $\{ s_3 \} \in \S_2.$
We may assume that ({\bf a1}) and ({\bf a3}) do not hold for $k=3$ or $k=4$.  

Suppose ({\bf a5}) holds for $k=3.$  Then $B_1^3 = \{ s_2 \}$ or $B_1^3 = \{ s_4 \}.$  In the former case, we have $A_1^3 = \{ s_1, s_4 \},$ and hence $A_1^3 + s_3 = \{ s_1, s_3, s_4 \} \in \S_1.$  However, since $\{ s_1, s_2, s_3 \} \in \S_1,$ it would follow that $\{s_1, s_3, s_4 \} \cup \{ s_1, s_2, s_3 \} = S \in \S,$ contradicting ({\bf S3}).  Thus $B_1^3 = \{ s_4 \}$ and $A_1^3 = \{ s_1, s_2 \}.$  We have that $B_1^3 + s_3 = \{ s_3, s_4 \} \in \S_1.$
However, given that $\{ s_2, s_4 \} \in \S_1,$ it follows by ({\bf S2}) that $\{ s_3, s_4 \} \cup \{ s_2, s_4 \} = \{ s_2,s_3,s_4 \} \in \S_1.$  Again, since $\{s_1, s_2, s_3 \} \in \S_1,$ it follows that $\{ s_2, s_3, s_4 \} \cup \{ s_1, s_2, s_3 \} = S \in \S_1,$ yielding a contradiction.  We conclude that ({\bf a5}) does not hold for $k=3.$  
By similar arguments, one can also show that ({\bf a5}) does not hold for $k=4$ either.

Suppose ({\bf a4}) holds for $k=4.$  Then there exists a subset $D' \in {\binom {S^4}2}$ and integers $i,j$ where $i+j = 3$ such that ${\binom {D'}i} \subseteq \S_1^4$ and ${\binom {D'}j} \subseteq \S_2^4.$  
We have that $D=D' + s_4 \in \S_1$. Given that $\{ s_1, s_2, s_3 \} \in \S_1,$ it follows by ({\bf S2}) that $S = D \cup \{ s_1, s_2, s_3 \} \in \S_1,$ a contradiction.  Thus ({\bf a4}) does not hold for $k=4$ and hence ({\bf a2}) holds for $k=4.$ Furthermore, since by ({\bf \ref{nona1.25}}), ({\bf a2}) holds for at most one of $k=3$ or $k=4,$ it must be the case that ({\bf a4}) holds for $k=3.$
As such, there exists a subset $D' \in {\binom {S^3}2}$ and integers $i,j$ where $i+j = 3$ such that ${\binom {D'}i} \subseteq \S_1^3$ and ${\binom {D'}j} \subseteq \S_2^3.$  
We have that $D=D' + s_3 \in \S_1$. Given that $\{ s_1, s_2, s_3 \} \in \S_1,$ if $D \ne \{ s_1, s_2, s_3 \},$ then by ({\bf S2}), $S = D \cup \{ s_1, s_2, s_3 \} \in \S_1,$ a contradiction.  Thus we must have that $D = \{ s_1, s_2, s_3 \},$ and thus $D' = \{ s_1, s_2 \}.$
If $j=1,$ then ${\binom {D'}1} \subseteq \S_2^3 \subseteq \S_2$.  Given $|D'| = 2 = n-2,$ the assertion holds in this case.  Thus we may assume that $j=2$ and $i=1.$  However, this means that $\{ s_1 \} \in \S_1^3,$ implying that $\{ s_1, s_3 \} \in \S_1.$  This in turn implies that $\{ s_3 \} \in \S_1$ (since $\{ s_1 \} \in \S_1$) yielding a contradiction.  
This completes the case for $n=4.$
\end{proof}

By ({\bf \ref{nona3}}), there exists $i \in \{ 1,2 \}$ and $T\in {\binom S{n-2}}$ for which ${\binom T1} \subseteq \S_i.$  Using similar reasoning as before, it suffices to prove the case where ${\binom T1} \subseteq \S_1$ (see the first paragraph of the proof of ({\bf \ref{nona3}})).  Thus we may assume ${\binom T1} \subseteq \S_1$ and moreover, $T = \{ s_1, \dots ,s_{n-2} \}$. 

Suppose first that ({\bf a1}) holds for $k=n-1$; that is, $\{ s_{n-1} \} \in \S_1.$  Then ${\binom {S^n}1} \subseteq \S_1$ and (by Observation \ref{obs2}), $S^n \in \S_1.$  We shall show that ({\bf a1}) - ({\bf a5}) do not hold for $k=n,$ violating ({\bf \ref{nona1}}).
Clearly ({\bf a1}) does not hold for $k=n$, for otherwise ({\bf S3}) is violated.  If ({\bf a2}) or ({\bf a3}) holds for $k=n$, then $S^n \in \S_2.$ In this case, ({\bf S4}) is violated.  
Suppose ({\bf a4}) holds for $k=n.$   

Then there exists $D' \in {\binom {S^n}{n-2}}$ and $1\le i \le n-2$ where ${\binom {D'}i} \subseteq \S_1^n$, and $D = D' + s_n \in \S_1.$  However, since $S^n \in \S_1,
$ it follows by ({\bf S2}) that $D \cup S^n = S \in \S_1,$ violating ({\bf S3}).  Thus ({\bf a4}) does not hold for $k=n.$  If ({\bf a5}) holds for $k=n,$ then there is a set $A_1^n \in {\binom {n-2}{\S_1^n}}$, implying that $D = A_1^n + s_n \in \S_1.$  Again, we have $D \cup S^n = S \in \S_1$, a contradiction. This shows that ({\bf a1}) - ({\bf a5}) do not hold for $k=n$ (a contradiction) and hence ({\bf a1}) can not hold for $k=n-1.$  By similar arguments, one can also show that ({\bf a1}) does not hold for $k=n.$
%
%
%

Suppose ({\bf a2}) holds for $k=n-1$. Then $S^{n-1} = \{ s_1, \dots ,s_{n-2}, s_n \} \in \S_2.$  We will show that ({\bf a4}) holds for $k=n$.  By ({\bf \ref{nona1.25}}), neither ({\bf a2}) nor ({\bf a3}) holds for $k=n.$  Suppose ({\bf a5}) holds for $k=n.$  Following a previous argument, we have that $\{ s_1, \dots ,n-2 \} \cap B_1^n = \emptyset.$  
Thus $B_1^n \subseteq \{ s_{n-1} \}$ and $n\le4.$ Given $n\ge 4,$ it follows that $n=4$ and
$B_1^4 = \{ s_3 \}$ and $A_1^4 = \{ s_1, s_2 \}.$  Thus $S^3 = \{ s_1, s_2, s_4 \} \in \S_1.$  Since for $i=1,2,$ $\{ s_i \} \in \S_1,$ it follows by Observation \ref{obs3} that ${\binom {S^3}1} \subseteq \S_1.$  However, this implies that $\{ s_4 \} \in \S_1,$ a contradiction.

It follows from the above that, assuming ({\bf a2}) holds for $k=n-1,$ ({\bf a4}) holds for $k=n.$  Thus there exists $D' \in {\binom {S^n}{n-2}}$ and integers $i,j,\ i+ j = n-1,$ such that ${\binom {D'}i} \subseteq \S_1^n$ and ${\binom {D'}j} \subseteq \S_2^n.$
Then $D = D' + s_n  \in \S_1.$  If $D' = \{ s_1, \dots ,s_{n-2} \},$ then $D' \in \S_1,$ (since ${\binom {D'}1} \subseteq \S_1$). It now follows by Observation \ref{obs3} that ${\binom D1} \subseteq \S_1.$ However, this implies that
$\{ s_n \} \in \S_1,$ a contradiction.
Thus $s_{n-1} \in D'.$  We have ${\binom {D'-s_{n-1}}1} \subseteq \S_1$ and $D' - s_{n-1} \in \S_1$.  Note that $D' \not\in \S_1;$ for otherwise, Observation \ref{obs3} would imply that ${\binom {D'}1} \subseteq \S_1,$ contradicting the fact that $\{ s_{n-1} \} \not\in \S_1.$

Suppose $i \le n-3$. Then ${\binom {D' - s_{n-1}}i} \subseteq \S_1^n.$ Thus for all $S' \in {\binom {D' - s_{n-1}}i},$ $S' \in \S_1$ and $S' + s_n \in \S_1.$  It follows by ({\bf S1}) that ${\binom {S'+s_n}i} \subseteq \S_1.$  This in turn implies that ${\binom {D' - s_{n-1} + s_n}i} \subseteq \S_1.$  By Observation \ref{obs2}, $D' - s_{n-1} + s_n \in \S_1.$  However, we also have that $\{ s_1, \dots ,s_{n-2} \} \in \S_1$ and thus $\{ s_1, \dots ,s_{n-2} \} \cup (D' - s_{n-1} + s_n) = S^{n-1} \in \S_1$.  Given that $S^{n-1} = D' - s_{n-1} + s_n + s_i,$ for some $i\in [n-2],$ it follows by Observation \ref{obs3} that ${\binom {S^{n-1}}i} \subseteq \S_1.$  By Observation \ref{obs2}, we have ${\binom {S^{n-1}}{i+1}} \subseteq \S_1.$ Since ${\binom {D'}j} \subseteq \S_2^n \subseteq \S_2$ and $S^{n-1} \in \S_2$ (since ({\bf a2}) holds for $k=n-1$) and $S^{n-1} - s_i = D',$ for some $i \in [n-2]$, it follows by 
Observation \ref{obs3} that ${\binom {S^{n-1}}j} \subseteq \S_2.$ However, we have ${\binom {S^{n-1}}{i+1}} \subseteq \S_1$ and ${\binom {S^{n-1}}j} \subseteq \S_2$ and $i+1 + j = n,$ in violation of ({\bf S4}). 
 %
%
%
%

From the above, we have $i= n-2$ and $j=1.$  Then $D' \in \S_1^n$ and ${\binom {D'}1} \subseteq \S_2.$ Let $A_1 = D' + s_n, \ A_2 = S^{n-1}, \ B_1 = S - s_{n-1} - s_n,$ and $B_2 = D'.$
Then by the above, $(A_1, A_2) \in {\binom {n-1}{\S_1}} \times {\binom {n-1}{\S_2}}$ and $A_1 \ne A_2.$  Furthermore, we have that for $i=1,2,\ {\binom {B_i}1} \subseteq \S_i.$  We also see that
$B_1 \cap B_2 = D' \cap \{ s_1, \dots ,s_{n-2} \} = A_1 \cap B_1 = A_2 \cap B_2.$ Thus in this case, the theorem is satisfied.

%
%
To finish the proof, we will show that no other options are possible. Suppose now that ({\bf a2}) does not hold for $k=n-1$, and we may assume the same is true for $k=n$.  Thus  ({\bf a3}) does not hold for $k=n-1$ or $k=n.$  

Suppose ({\bf a4}) holds for $k=n-1.$  
Then there exists $D' \in {\binom {S^{n-1}}{n-2}}$ and integers $i,j,\ i+ j = n-1,$ such that ${\binom {D'}i} \subseteq \S_1^{n-1}$ and ${\binom {D'}j} \subseteq \S_2^{n-1} \subseteq \S_2.$ Then $D = D' +  s_{n-1}  \in \S_1.$
As before, $D' \ne \{ s_1, \dots ,s_{n-2} \}.$ 
Thus $s_n \in D'$ and we may assume without loss of generality that $D' = \{ s_1, \dots ,s_{n-3}, s_n \}.$
By ({\bf \ref{nona1.5}}), ({\bf a4}) does not hold for $k=n.$  Thus ({\bf a5}) holds for $k=n$ and there exist 
$(A_1^n, A_2^n) \in {\binom {n-2}{\S_1^n}}\times {\binom {n-2}{\S_2^n}},$ $A_1^n \ne A_2^n$,
and $\{ B_1^n, B_2^n \} \subseteq {\binom {S^n}{n-3}}$ where for $i=1,2$, $B_i^n \cap A_i^n = B_1^n \cap B_2^n  \in  {\binom {A_1^n \cap A_2^n}{n-4}}$ and
${\binom{B_i^n}1} \subseteq \S_i^n.$
Arguing as before, we have $B_1^n \cap \{ s_1, \dots ,s_{n-2} \} = \emptyset.$  This in turn implies that $B_1^n = \{ s_{n-1} \}$ and hence $n=4.$  Furthermore, we have that $A_1^n = A_1^4 = \{ s_1, s_2 \},$ implying that $\{ s_1, s_2, s_4 \} \in \S_1.$  However,
we also have that $D = \{ s_1, s_3, s_4 \} \in \S_1.$  It follows by ({\bf S2}) that $S=D \cup \{ s_1,s_2, s_4 \} \in \S_1,$ violating ({\bf S3}). Thus ({\bf a4}) does not hold for  $k=n-1$ and the same holds for $k =n.$

From the above, ({\bf a5}) must hold for both $k=n-1$ and $k=n$. 
Using similar arguments as above, one can show that $n=4,$ $B_1^3 = \{ s_4 \},$ $A_1^3 = \{ s_1, s_2 \},$ $B_1^4 = \{ s_3 \},$ and $A_1^4 = \{s_1, s_2 \}.$  We have $A_1^3 + s_3 = \{ s_1, s_2, s_3 \} \in \S_1$ and $A_1^4 + s_4 = \{ s_1, s_2, s_4 \} \in \S_1.$  It now follows by ({\bf S2}) that $\{ s_1, s_2, s_3 \} \cup \{ s_1, s_2, s_4 \} = S \in \S_1,$ contradicting ({\bf S3}).
%
This completes the proof of the theorem.
\end{proof}

\section{Proof of Theorem \ref{the-main}}

Let $M$ be a paving matroid where $\gamma(M) = \beta(E(M))$ and $|E(M)| =n.$  

\subsection{The case $r(M) =2$}

Suppose $r(M)=2.$ We shall prove by induction on $n$ that $M$ is cyclically orderable.
%
Theorem \ref{the-main} is seen to be true when $n=2.$  Assume that it is true when $n=m-1\ge 2.$  We shall prove that it is also true for $n=m.$  Assume that $M$ is a paving matroid where $r(M) =2$, $|E(M)| =m$ and $\gamma(M) = \beta(E(M)) = \frac m2.$
For all elements $e\in E(M)$, let $X_e$ denote the parallel class containing $e$ and let $m(e) = |X_e|.$  Then for all $e\in E(M),$ $\beta(X_e) = m(e) \le \gamma(M) = \frac m2.$
If there are elements $e\in E(M)$ for which $m(e) = \frac m2,$ then choose $f$ to be one such element.  If no such elements exist, then let $f$ be any element in $M.$
Let $M' = M\backslash f.$ 
Suppose there exists $X \subseteq E(M')$ for which $\beta(X) > \frac {m-1}2 = \beta(E(M').$  Then clearly $r(X) =1.$ Thus $X \subseteq X_{g}$ for some $g \in E(M').$  Given that $m(g) \le \frac m2,$ it follows that $X = X_g$ and $m(g) = \frac m2.$  By the choice of $f$, we also have $m(f) = \frac m2.$  Then $E(M) = X_f \cup X_g$  and $E(M) = m = 2\ell$, for some integer $\ell.$  Now let $e_1e_2 \cdots e_m$ be an ordering of $E(M)$ where for all $i$, $e_i \in X_f$, if $i$ is odd, and $e_i\in X_g$, if $i$ is even.  This gives a cyclic ordering for $M$.  Thus we may assume that $\gamma(M') = \beta(E(M')) = \frac {m-1}2.$
By assumption, there is a cyclic ordering for $M'$, say
$e_1e_2 \cdots e_{m-1}.$  Since $m(f) \le \frac m2,$ there exists $i\in [m-1]$ such that $\{ e_i, e_{i+1} \} \cap X_f = \emptyset.$
Consequently, $e_1 \cdots e_i f e_{i+1} \dots e_{m-1}$ is seen to be a cyclic ordering for $M$.  The proof now follows by induction.

\subsection{The case where $|E(M)| \le 2r(M)+1$}

Suppose $|E(M)| \le 2r(M) +1.$  As mentioned earlier, if $|E(M)| = 2r(M)+1,$ then $|E(M)|$ and $r(M)$ are relatively prime and hence it follows by Theorem \ref{the-VanTho} that $M$ has a cyclic ordering.
Thus we may assume that $|E(M)| \le 2r(M).$  It now follows by Theorem \ref{the0} that there are bases $A$ and $B$ for which $A \cup B = E(M).$

The following is a well-known conjecture of Gabow \cite{Gab}.

\begin{conjectureplus}{Gabow}
Suppose that $A$ and $B$ are bases of a matroid $N$ of rank $r.$  Then there are orderings $a_1a_2 \cdots a_r$ and $b_1b_2 \cdots b_r$ of the elements of $A$ and $B$, respectively, such that for $i =1, \dots ,r-1,$
$\{ a_1, \dots ,a_i, b_{i+1}, \dots ,b_r \}$ and $\{ a_{i+1}, \dots ,a_r, b_1 \dots b_i \}$ are bases.\label{con-Gab}
\end{conjectureplus}

We observe that in the special case of Conjecture \ref{con-Gab} where $E(N)$ is the union of two bases, the conjecture implies that $N$ has a cyclic ordering.  
In \cite{BerSch}, the authors verify, among other things, the above conjecture for {\it split matroids}, a class of matroids which includes all paving matroids.
Given that the above conjecture is true for split matroids (and hence also paving matroids) and $E(M) = A \cup B,$ it follows that $M$ has a cyclic ordering.

\subsection{The case where $|E(M)| \ge 2r(M) +2$ and $r(M) \ge 3$.}

In this section, we shall assume that $|E(M)| \ge 2r(M) +2$ and $r(M) \ge 3$.
By Proposition \ref{pro1}, there exists a basis $S$ of $M$ for which $\gamma(M\backslash S) = \beta(E(M)-S)$ and $r(M\backslash S) = r(M).$  Let $r = r(M)$ and let $S = \{ s_1, \dots ,s_r \}.$ Let $M' = M \backslash S$ and let $m = |E(M')| = n-r.$ 
By assumption, $M'$ is cyclically orderable and we will assume that $e_1e_2 \cdots e_m$ is a cyclic ordering. 
Our goal is to show that the cyclic ordering for $M'$ can be extended to a cyclic ordering of $M$.  
To complete the proof of Theorem \ref{the-main}, we need only prove the following:

\begin{proposition}
There exists $i\in [m]$ and a permutation $\pi$ of $[r]$ such that\\ $e_1e_2 \cdots e_i s_{\pi(1)}s_{\pi(2)} \cdots s_{\pi(r)} e_{i+1} \cdots e_m$ is a cyclic ordering of $M.$ 
\label{pro2}
\end{proposition}

\begin{proof}
Assume to the contrary that for all $i\in[m]$ and for all permutations $\pi$ of $[r],$ $e_1e_2 \cdots e_i s_{\pi(1)}s_{\pi(2)} \cdots s_{\pi(r)} e_{i+1} \cdots e_m$ is not a cyclic ordering of $M.$
For all $j\in [m],$ we shall define a pair $(\H_1^j, \H_2^j)$, where for $i=1,2,$ $\H_i^j \subseteq 2^S.$ 
 Let $x_1^j = e_{j-1}, \ x_2^j = e_{j-2}, \dots ,x_{r-1}^j = e_{j-r+1},$ and let $y_1^j = e_{j}, \ y_{2}^j =e_{j+1}, \dots ,y_{r-1}^j = e_{j+r-2}$ where for all integers $k$, we define $e_k := e_\ell$ where
 $$\ell := \left\{ \begin{array}{lr} k \ mod\ m & \mathrm{if}\ k\ mod\ m \ne 0\\ m  & \mathrm{otherwise.} \end{array} \right. $$ 
 Let $X^j = \{ x_1^j, \dots ,x_{r-1}^j \}$ and $Y^j = \{ y_1^j, \dots ,y_{r-1}^j \}$.  

Let $\pi$ be a permutation of $[r]$. By assumption,\\ $e_1 \cdots e_{j-1} s_{\pi(1)}s_{\pi(2)} \cdots s_{\pi(r)} e_j \cdots e_m$ is not a cyclic ordering for $M.$  Then there exists $i \in [r-1]$ such that either $\{ x_1^j, \dots ,x_i^j \} \cup \{ s_{\pi(1)}, \dots ,s_{\pi(r-i)} \}$ is dependent or $\{ y_1^j, \dots ,y_i^j \} \cup \{ s_{\pi(i+1)}, \dots ,s_{\pi(r)} \}$ is dependent.   Since the smallest circuit has size $r$, this means that either $\{ x_1^j, \dots ,x_i^j \} \cup \{ s_{\pi(1)}, \dots ,s_{\pi(r-i)} \}$ or $\{ y_1^j, \dots ,y_i^j \} \cup \{ s_{\pi(i+1)}, \dots ,s_{\pi(r)} \}$ is a circuit.   Let $\C_1^j$ be the set of all $r$-circuits which occur in the former case, and let $\C_2^j$ be the set of all $r$-circuits occurring in the latter case.  That is,
$\C_1^j$ is the set of all $r$-circuits $C$ where for some $i\in [r-1],$ $\{ x_1^j, \dots , x_i^j \} \subset C \subset \{ x_1^j, \dots ,x_i^j\} \cup S$, and $\C_2^j$ is set of all $r$-circuits $C$ where for some $i\in [r-1],$  $\{ y_1^j, \dots ,y_i^j \} \subset C \subseteq \{ y_1^j, \dots ,y_i^j \} \cup S.$  
For $i=1,2,$ let $\H_i^j = \{ C \cap S \ \big| \ C \in \C_i^j \}.$ 

\begin{noname} For all $j,$ the pair $(\H_1^j, \H_2^j)$ is an $S$-pair which is order-consistent. \label{nonapro2} \end{noname}

\begin{proof}
It suffices to prove the assertion for $j=1.$  For convenience, we let $x_i = x_i^1$ $y_i = y_i^1, \  i = 1, \dots ,r-1.$  Furthermore, we let $X = X^1, \ Y= Y^1,$ $\H_1 = \H_1^1$, $\H_2 = \H_2^1,$ $\C_1 = \C_1^1,$ and $\C_2 = \C_2^1.$
It follows from the definition of $(\H_1, \H_2)$ that it is order-consistent.  We need only show that it is an $S$-pair.
Suppose $A,B \in \H_1$ where $|A| = |B|+1$ and $B \subset A.$  Then for some $i \in [r-1],$ $C_1 = A \cup \{ x_1, \dots ,x_i \} \in \C_1$ and $C_2 = B \cup \{ x_1, \dots ,x_{i+1} \} \in \C_1.$  Let $x \in B.$  Then $x\in C_1 \cap C_2$ and hence by the circuit elimination axiom there is a circuit $C \subseteq (C_1 \cup C_2)-x = (A -x) \cup \{ x_1, \dots , x_{i+1} \}.$  Thus $C = (A -x) \cup \{ x_1, \dots , x_{i+1} \}$ and hence $A-x \in \H_1.$  Since this applies to any element $x\in B$, it follows that 
${\binom A{|B|}} \subseteq \H_1.$  The same arguments can be applied to $\H_2$.  Thus ({\bf S1}) holds.

To show that ({\bf S2}) holds, suppose $A,B \in \H_1$ where $|A| = |B|$ and $|A\cap B| = |A|-1.$  There exists $i\in [r]$ such that $C_1 = \{ x_1, \dots ,x_i \} \cup A \in \C_1$ and  $C_2 = \{ x_1, \dots ,x_i \} \cup B \in \C_1.$  By the circuit elimination axiom, there exists a circuit
$C \subseteq (C_1 \cup C_2) - x_i = (A\cup B) \cup \{ x_1, \dots ,x_{i-1} \}$.  Thus $C = (A\cup B) \cup \{ x_1, \dots ,x_{i-1} \}$ is a circuit and hence $A\cup B \in \H_1.$  The same reasoning applies if $A,B \in \H_2.$  Thus ({\bf S2}) holds.

To show that ({\bf S3}) holds, suppose ${\binom S1} \subseteq \H_1.$ Then for $i=1, \dots ,r-1,\  C_i = X \cup \{ s_i \}$ is a circuit, and consequently, $S \subseteq \mathrm{cl}(X).$  However, this is impossible since $|X| = r-1 < r(S) = r.$  Thus  ${\binom S1} \not\subseteq \H_1$ and likewise,  ${\binom S1} \not\subseteq \H_2.$ Also, we clearly have that for $i=1,2,$ $S \not\in \H_i$ since $S$ is a base of $M.$  Thus ({\bf S3}) holds.

Lastly, to show that ({\bf S4}) holds, let $S' = S-s_r.$   Suppose first that ${\binom {S'}{r-1}} \subseteq \H_1$ and ${\binom {S'}1} \subseteq \H_2.$   Then
$S' \in \H_1$ and hence $S' + x_1 \in \C_1.$  Also, for all $i \in [r-1],$ $Y + s_i \in \C_2$.  Thus $x_1 \in \mathrm{cl}(S')$ and $S' \subseteq \mathrm{cl}(Y)$.  Given that $S'$ is independent and $|S'| = |Y| = r-1,$ it follows that $\mathrm{cl}(S') = \mathrm{cl}(Y).$  However, this implies that $Y + x_1 = \{ x_1, y_1, \dots ,y_{r-1} \} = \{ e_m, e_1, \dots ,e_{r-1} \} \subseteq \mathrm{cl}(S'),$ which contradicts the assumption that $\{ e_m, e_1, \dots ,e_{r-1} \}$ is a basis of $M.$ 

Suppose now that for some $k \in [r-2]$, ${\binom {S'}k} \subseteq \H_1$ and ${\binom {S'}{r-k}} \subseteq \H_2$. We claim that $\{ x_1, \dots ,x_{r-k} \} \cup \{ y_1, \dots ,y_k \} \subseteq \mathrm{cl}(S')$.  Following the proof of Observation \ref{obs2}, we have that for $j = k, \dots ,r-1$, ${\binom {S'}j}  \subseteq \H_1.$  In particular, $S' \in \H_1,$ and hence $C_1 = S' +  x_1  \in \C_1.$  This implies that $x_1 \in \mathrm{cl}(S').$  However, seeing as ${\binom {S'}{r-2}} \subseteq \H_1,$ we have that $C_2 = (S' - s_{r-1}) \cup \{ x_1, x_2 \} \in \C_1.$  Given that $x_1 \in \mathrm{cl}(S'),$ it follows that $x_2 \in \mathrm{cl}(S').$  Continuing, we see that $\{ x_1, \dots , x_{r-k} \} \subseteq \mathrm{cl}(S').$  By similar arguments, it can be shown that $\{ y_1, \dots ,y_k \} \subseteq \mathrm{cl}(S').$ Thus proves our claim.  It follows that $r(\{ x_1, \dots ,x_{r-k} \} \cup \{ y_1, \dots ,y_k \} ) \le r-1.$
However, this is impossible since by assumption  $\{ x_1, \dots ,x_{r-k} \} \cup \{ y_1, \dots ,y_k \}$ is a basis.  Thus no such $k$ exists. More generally, the same arguments can be applied to any $j\in [r]$ and $S' = S-s_j.$  Thus ({\bf S4}) holds.
\end{proof}
  
By ({\bf \ref{nonapro2}}), for all $j \in [m],$ $(\H_1^j,\H_2^j)$ is an $S$-pair which is order-consistent.  Thus it follows by Theorem \ref{the1}, that for all $j\in [m],$ there exists $(A_1^j, A_2^j) \in {\binom {r-1}{\H_1^j}}\times {\binom {r-1}{\H_2^j}},$ $A_1^j \ne A_2^j,$
and $\{ B_1^j, B_2^j \} \subseteq {\binom S{r-2}}$ where for $i=1,2$, $B_i^j \cap A_i^j = B_1^j \cap B_2^j  \in  {\binom {A_1^j \cap A_2^j}{r-3}}$ and
${\binom{B_i^j}1} \subseteq \H_i^j.$

Suppose $r>4.$  Given that $|B_1^1| = |B_1^2| = r-2,$ it follows that there exists $s_i \in B_1^1 \cap B_1^2.$  Then $\{ s_i \} \in \H_1^1 \cap \H_1^2$ and consequently, $C_1 = \{ s_i, e_{m-r+2}, \dots ,e_{m} \}$ and $C_2 = \{ s_i, e_{m-r+3}, \dots ,e_m, e_1 \}$ are distinct circuits in $M.$  By the circuit elimination axiom, there exists a circuit
$C \subseteq (C_1 \cup C_2) - s_i = \{ e_{m-r+2}, \dots ,e_m, e_1 \}.$  However, this is impossible since by assumption, $\{ e_{m-r+2}, \dots , e_m, e_1 \}$ is a basis.  Therefore, $r\le 4.$

Suppose $r=3$.  Without loss of generality, we may assume that $A_1^1 = \{ s_1, s_2 \}$, $B_1^1 = \{ s_3 \}$, $A_2^1 = \{ s_2, s_3 \}$, and  $B_2^1= \{ s_1 \}$. Then $\{ s_3, e_m, e_{m-1} \}$ and $\{ s_1, e_1, e_2 \}$ are circuits.
We have that $B_1^2 \ne \{ s_3 \}$ and $B_2^2 \ne \{ s_1 \};$ for if $B_1^2 = \{ s_3 \}$, then $B_1^1 = B_2^2 = \{ s_3 \}$ and it follows that $\{ s_3, e_{m-1}, e_m \}$ and $\{ s_3, e_1, e_m \}$ are circuits, implying that  $\{ e_{m-1}, e_m, e_1 \}$ is a circuit -- a contradiction. Similar reasoning applies if $B_2^2 = \{ s_1 \}.$  Suppose that $B_1^2 = \{ s_1 \} .$  Then $\{ s_1, e_1, e_m \}$ is a circuit.  However, seeing as $\{ s_1, e_1, e_2 \},$ is a circuit (since $B_2^1 = \{ s_1 \}$),  it follows that $\{ s_1, e_1, e_2 \} \cup \{ s_1, e_1, e_m \} - s_1 = \{ e_m, e_1, e_2 \}$ is a circuit, which is false since by assumption $\{ e_m, e_1, e_2 \}$ is a basis.  Thus $B_1^2 \ne \{ s_1 \}.$ 
Given that $B_1^2 \ne \{ s_3 \},$ it follows that 
$B_1^2 = \{ s_2 \}$ and $A_1^2 = \{ s_1, s_3 \}.$  Since $B_2^2 \ne \{ s_1 \},$ it follows that $B_2^2 = \{ s_3 \}$ and $A_2^2 = \{ s_1, s_2 \}.$  Since $A_1^1= A_2^2 = \{ s_1, s_2 \},$ it follows that $\{ s_1, s_2, e_m \}$ and $\{ s_1, s_2, e_2 \}$ are circuits.  Furthermore, since $B_1^2 = \{ s_2 \},$ it follows that $\{ s_2, e_1, e_m \}$ is a circuit.  It is now seen that $\{ e_m, e_1, e_2 \} \subseteq \mathrm{cl}(\{ s_1, s_2 \} ),$ which contradicts the  assumption that  $\{ e_m, e_1, e_2 \}$ is a basis.

Lastly, suppose $r=4$.  Suppose $s_i \in B_1^1 \cap B_1^2.$  Then $\{ s_i, e_{m-2}, e_{m-1}, e_m \}$ and $\{ s_i, e_{m-1}, e_m, e_1 \}$ are circuits and 
hence $\{ e_{m-2}, e_{m-1},e_m, e_1 \}$ is also a circuit, contradicting our assumptions.  Thus $B_1^1 \cap B_1^2 = \emptyset$ and similarly, $B_2^1 \cap B_2^2 = \emptyset.$  More generally, for all $i \in \{ 1,2 \}$ and $j\in [m], \ B_i^j \cap B_i^{j+1} = \emptyset.$ 
Since for all $i \in \{1,2 \}$, $|B_i^1| = |B_i^2| = 2$ it follows that for all $i \in \{ 1,2 \}$, $j\in [m],$ $B_i^j \cup B_i^{j+1} = S.$ Without loss of generality, we may assume $B_1^1 = \{ s_1, s_2 \}$ and $B_1^2 = \{ s_3, s_4 \}.$  
Note that $B_1^1 = \{ s_1, s_2 \}$ means that $\{ s_1, s_2 \} \subset A_2^1$ and so $A_2^1 = \{ s_1, s_2, s_3 \}$ or 
$\{ s_1, s_2, s_4 \}$.  Given that $B_1^1 \not\subseteq A_1^1 \cap A_2^1,$ irregardless of whether $A_2^1$ is the former or latter we have that $A_1^1 = \{ s_1, s_3, s_4 \}$ or $\{ s_2, s_3, s_4 \}$.  However, since the indexing of the elements of $S$ is essentially arbitrary,  one can assume that $A_2^1$ is any one of the first two choices and $A_1^1$ is any one of the latter two choices.
Thus we may assume without loss of generality that  $A_1^1 = \{ s_2, s_3, s_4 \}$ and $A_2^1 = \{ s_1, s_2, s_4 \}.$  Since for all $i \in \{ 1,2 \}$, $j\in [m],$ $B_i^j \cup B_i^{j+1} = S,$ it follows that $B_1^1 = B_1^3 = \cdots = \{ s_1, s_2 \}$ and 
$B_1^2 = B_1^4 = \cdots = \{ s_3, s_4 \}.$  In particular, $m$ must be even.  Corresponding, for $i = 1, 3, \dots $, $A_1^i = \{ s_1, s_3, s_4 \}$ or $\{ s_2, s_3, s_4 \}$ and for $i = 2, 4, \dots $ $A_2^i = \{ s_1, s_2, s_3 \}$ or $\{ s_1, s_2, s_4 \}.$


Given that $B_1^1 = \{ s_1, s_2 \}$ and ${\binom {B_1^1}1} \subseteq \H_1^1,$ it follows that\\ $\{ s_1, e_{m-2}, e_{m-1}, e_m\}$ and $\{ s_2, e_{m-2}, e_{m-1}, e_m\}$ are circuits. Thus $\{ s_1, s_2 \} \subset \mathrm{cl}(\{ e_{m-2}, e_{m-1}, e_m \} ).$  By the above, we have that $B_1^{m} = \{ s_3, s_4 \}$ and either $A_1^{m} = \{ s_1, s_2, s_3 \}$ or $A_1^{m} = \{ s_1, s_2, s_4 \}.$  Suppose the former holds.  Then $\{s_1, s_2, s_3, e_{m-1} \}$ is a circuit.  Consequently, 
$s_3 \in  \mathrm{cl}(\{ e_{m-2}, e_{m-1}, e_m \} ).$  However, since $B_1^2 = \{ s_3, s_4 \} \in \H_1^2,$ it follows that $\{ s_3, e_{m-1}, e_m, e_1 \}$ and $\{ s_4, e_{m-1}, e_m, e_1 \}$  are circuits. By the circuit elimination axiom,\\ $\{ s_3, s_4, e_{m-1}, e_m \}$ is a circuit and hence $s_4 \in \mathrm{cl}(\{ e_{m-2}, e_{m-1}, e_m \} ).$  However, it now follows that $\{ s_1, s_2, s_3, s_4 \} \subset \mathrm{cl}(\{ e_{m-2}, e_{m-1}, e_m \} )$, yielding a contradiction.  If instead, $A_1^{m} = \{ s_1, s_2, s_4 \},$ then similar arguments yield a contradiction. 
This concludes the case for $r=4.$
\end{proof}

\end{document}